\documentclass{amsart} 
\usepackage[all]{xy}
\usepackage{calc}
\usepackage{amsmath}
\usepackage{graphicx}
\usepackage{enumerate}
\begin{document}

\newcommand{\mmbox}[1]{\mbox{${#1}$}}
\newcommand{\proj}[1]{\mmbox{{\mathbb P}^{#1}}}
\newcommand{\affine}[1]{\mmbox{{\mathbb A}^{#1}}}
\newcommand{\Ann}[1]{\mmbox{{\rm Ann}({#1})}}
\newcommand{\caps}[3]{\mmbox{{#1}_{#2} \cap \ldots \cap {#1}_{#3}}}
\newcommand{\N}{{\mathbb N}}
\newcommand{\Z}{{\mathbb Z}}
\newcommand{\R}{{\mathbb R}}
\newcommand{\K}{{\mathbb K}}
\newcommand{\p}{{\mathbb P}}
\newcommand{\A}{{\mathcal A}}
\newcommand{\RJ}{{\mathcal R}/{\mathcal J}}
\newcommand{\C}{{\mathbb C}}
\newcommand{\CR}{C^r(\hat P)}
 \newcommand{\aff}{\mathop{\rm aff}}
\newcommand{\arrow}[1]{\stackrel{#1}{\longrightarrow}}
\newcommand{\coker}{\mathop{\rm coker}\nolimits}
\sloppy
\newtheorem{defn0}{Definition}[section]
\newtheorem{prop0}[defn0]{Proposition}
\newtheorem{conj0}[defn0]{Conjecture}
\newtheorem{thm0}[defn0]{Theorem}
\newtheorem{lem0}[defn0]{Lemma}
\newtheorem{rmk0}[defn0]{Remark}
\newtheorem{corollary0}[defn0]{Corollary}
\newtheorem{example0}[defn0]{Example}
\newenvironment{defn}{\begin{defn0}}{\end{defn0}}
\newenvironment{prop}{\begin{prop0}}{\end{prop0}}
\newenvironment{conj}{\begin{conj0}}{\end{conj0}}
\newenvironment{thm}{\begin{thm0}}{\end{thm0}}
\newenvironment{lem}{\begin{lem0}}{\end{lem0}}
\newenvironment{cor}{\begin{corollary0}}{\end{corollary0}}
\newenvironment{exm}{\begin{example0}\rm}{\end{example0}}
\newenvironment{rmk}{\begin{rmk0}\rm}{\end{rmk0}}

\newcommand{\defref}[1]{Definition~\ref{#1}}
\newcommand{\propref}[1]{Proposition~\ref{#1}}
\newcommand{\thmref}[1]{Theorem~\ref{#1}}
\newcommand{\lemref}[1]{Lemma~\ref{#1}}
\newcommand{\corref}[1]{Corollary~\ref{#1}}
\newcommand{\exref}[1]{Example~\ref{#1}}
\newcommand{\secref}[1]{Section~\ref{#1}}
\newcommand{\poina}{\pi({\mathcal A}, t)}
\newcommand{\poinc}{\pi({\mathcal C}, t)}
\newcommand{\std}{Gr\"{o}bner}
\newcommand{\jq}{J_{Q}}
\def\B{Bernstein-B\'ezier}
\title {Subdivision and spline spaces}

\author{Hal Schenck}
\thanks{H. Schenck is partially supported by  NSF grant \#1312071}\address{Schenck: Mathematics Department \\ University of Illinois Urbana-Champaign\\
  Urbana \\ IL 61801\\ USA}
\email{schenck@math.uiuc.edu}

\author{Tatyana Sorokina}
\thanks{T. Sorokina is partially supported by a grant from the Simons Foundation \#235411 }\address{Sorokina: Department of Mathematics \\ Towson University\\
  Towson \\ MD 21252\\ USA}
\email{tsorokina@towson.edu}

\subjclass[2000]{Primary 41A15, Secondary 13D40, 52C99} \keywords{spline, dimension formula}

\begin{abstract}
\noindent A standard construction in approximation theory is mesh
refinement. For a simplicial or polyhedral mesh $\Delta \subseteq
\R^k$, we study the subdivision $\Delta'$ obtained by subdividing a
maximal cell of $\Delta$. We give sufficient conditions for the module
of splines on $\Delta'$ to split as the direct sum of splines on $\Delta$ and splines on the subdivided cell. As a consequence, we obtain dimension formulas and explicit bases for several commonly used subdivisions and their multivariate generalizations.
\end{abstract}
\maketitle
\vskip -.1in
\vskip -.1in

\section{Introduction}\label{intro}
Splines are fundamental objects in approximation theory, computer aided  geometric 
design and modeling, and  the finite element method for solving PDEs. Starting 
 with a simplicial or polyhedral complex $\Delta$ (or mesh) which 
partitions a region in $\R^k$, it may be the case that the mesh is 
too coarse for the specific application. So a natural approach is 
to refine the mesh via subdivision. 

We use $\Delta$ to denote a $k$-dimensional 
simplicial complex in $\R^k$, $\Delta_i$ the set of $i$-dimensional faces, and
$\Delta_i^0$ the set of interior $i$-dimensional faces; all $k$-dimensional faces are considered interior so $\Delta_k = \Delta_k^0$. Although we work in
the simplicial setting, all our results generalize easily to the 
polyhedral case. We analyze a special type of subdivision, 
where the original mesh $\Delta$ is modified by subdividing a single 
maximal cell $\sigma \in \Delta_k$. For the resulting object $\Delta'$ to 
be a complex, it is necessary that any modifications made to 
the boundary of $\sigma$ occur only on $\sigma \cap \partial(\Delta)$. 

The principal technique we use is the homological approach introduced by Billera
in \cite{b}, combined with the observation that  if $\widehat \Delta$ is the cone over $\Delta$, then $S^r(\widehat \Delta)$ is a graded module over the polynomial ring $R=\mathbb{R}[x_0,\ldots,x_k]$, and the dimension 
of $S^r_d(\Delta)$ is the dimension of the $d^{th}$ graded piece of 
$S^r(\widehat \Delta)$. 

Our main result is a sufficient condition for the set of 
splines on $\widehat \Delta'$ to split as the direct sum of splines on $\widehat \Delta$ and 
splines on the subdivided cell $\widehat \Delta''$:
\[
S^r(\widehat\Delta') \simeq S^r(\widehat\Delta) \bigoplus \Big(S^r(\widehat\Delta'')/\R[x_0,\ldots, x_k]\Big),
\]
where quotienting of the second summand by $\R[x_0,\ldots, x_k]$ 
corresponds to eliminating the splines defined by the same polynomial  on all maximal cells of $\Delta''$.
As a consequence, we obtain dimension 
formulas and explicit bases for several commonly used subdivisions, 
their multivariate generalizations,  as well as on various 
intermediate subdivisions. For these subdivisions, $S^r(\Delta')$ is free,
and a generalization \cite{s1} of Schumaker's lower bound for the
planar case \cite{schumaker} gives the correct dimension.

\section{Homology and subdivisions}\label{homology}
\noindent We work with the modification of Billera's complex introduced in \cite{ss}. 
Throughout this paper, our basic references are~\cite{ls} for 
splines and~\cite{E} for  algebra.
\begin{defn}\label{compC} For a full-dimensional simplicial complex 
$\Delta \subseteq \R^k$, 
let $\RJ(\Delta)$ be the complex of $R=\R[x_0,\ldots,x_k]$ modules, 
with differential $\partial_i$ the usual boundary operator in relative
(modulo boundary) homology.
\[
0 \longrightarrow \bigoplus\limits_{\sigma \in \Delta_k} R 
\stackrel{\partial_k}{\longrightarrow} \bigoplus\limits_{\tau \in \Delta_{k-1}^0} R/J_{\tau}  
\stackrel{\partial_{k-1}}{\longrightarrow} \bigoplus\limits_{\psi \in \Delta_{k-2}^0} R/J_{\psi}  
\stackrel{\partial_{k-2}}{\longrightarrow} \ldots
\stackrel{\partial_{1}}{\longrightarrow} \bigoplus\limits_{v \in \Delta_{0}^0} R/J_{v}\longrightarrow 0,
\]
where for an interior $i$-face $\gamma \in \Delta_i^0$, we define 
\[
J_\gamma = \langle l_{\widehat \tau}^{r+1} \mid \gamma \subseteq \tau \in \Delta_{k-1} \rangle.
\]
\end{defn}
\noindent The ideal $J_\gamma$ is generated by $r+1^{st}$ powers of 
homogenizations $l_{\widehat\tau}$ of linear forms $l_\tau$ whose vanishing defines the affine span of faces
$\tau$  containing $\gamma$. The top homology module of $\RJ(\Delta)$ 
computes splines of smoothness $r$ on $\widehat{\Delta}$.

\begin{thm}\cite{s1}\label{freeness}
If $\Delta$ is a 
topological $k$-ball, then the module
$S^r(\widehat \Delta)$ is free iff $H_i(\RJ(\Delta))=0$ for all $i <
k$. In this case, 
\[
\dim S^r(\widehat \Delta)_d = \sum\limits_{i=0}^k (-1)^i \dim
(\RJ_{k-i})_d,
\]
where $\RJ_{k-i}=\bigoplus_{\psi\in\Delta^0_{k-i}}R/J_\psi$.
\end{thm}

\subsection{Split subdivisions} 
Our strategy is to relate splines on a simplicial complex  $\Delta$ to splines 
on a complex $\Delta'$ obtained by subdividing some $\sigma \in \Delta_k$. 
\begin{defn}\label{splitSD}
Let $\Delta \subseteq \R^k$ be a $k$-dimensional simplicial complex, 
$\sigma \in \Delta_k$, and $\Delta''$ a subdivision of $\sigma$,
such that $\partial(\sigma) = \partial(\Delta'')$ on $\Delta^0$. 
Then the resulting subdivision $\Delta'$ is again a
simplicial complex, and we call the subdivision a {\em\bf simple}
subdivision. For each $i$-face $\gamma \in \Delta'_i$, let $J(\Delta')_\gamma$ 
denote the ideal in Definition~\ref{compC}. We call a simple 
subdivision $\Delta'$ {\em\bf split} if for every 
$\gamma \in \partial(\Delta'')_i$ but not in $\partial(\Delta')$,
\[
J(\Delta')_\gamma = J(\Delta)_\gamma.
\]
\end{defn}
Note that Definition~\ref{splitSD} imposes no conditions on faces of $\Delta'' \cap \partial(\Delta')$.
The following example illustrates simple and split subdivisions.
\begin{exm}\label{r1}
We start with $\Delta$ depicted in Figure~\ref{delta}, and subdivide the interior triangle into three subtriangles
as in Figure~\ref{sigma}. Both subdivision $\Delta'$ in Figure~\ref{delta1} and $\widetilde\Delta'$ in Figure~\ref{delta2} are simple. 
Moreover, when $r=1$, both  $\Delta'$ and $\widetilde\Delta'$ are split, 
because when $r=1$ as soon as there are three distinct slopes at a vertex, $J(v)$ is the square
of the ideal of the vertex. However, when $r\geq 2$, only $\Delta'$ in Figure~\ref{delta1}, 
where the new edge has the same slope as an existing edge, is a split
subdivision.
\begin{figure}
\begin{minipage}{0.45\linewidth}
\vskip 10pt
\centering
 \includegraphics[keepaspectratio=true,  width=38mm, height=47mm]{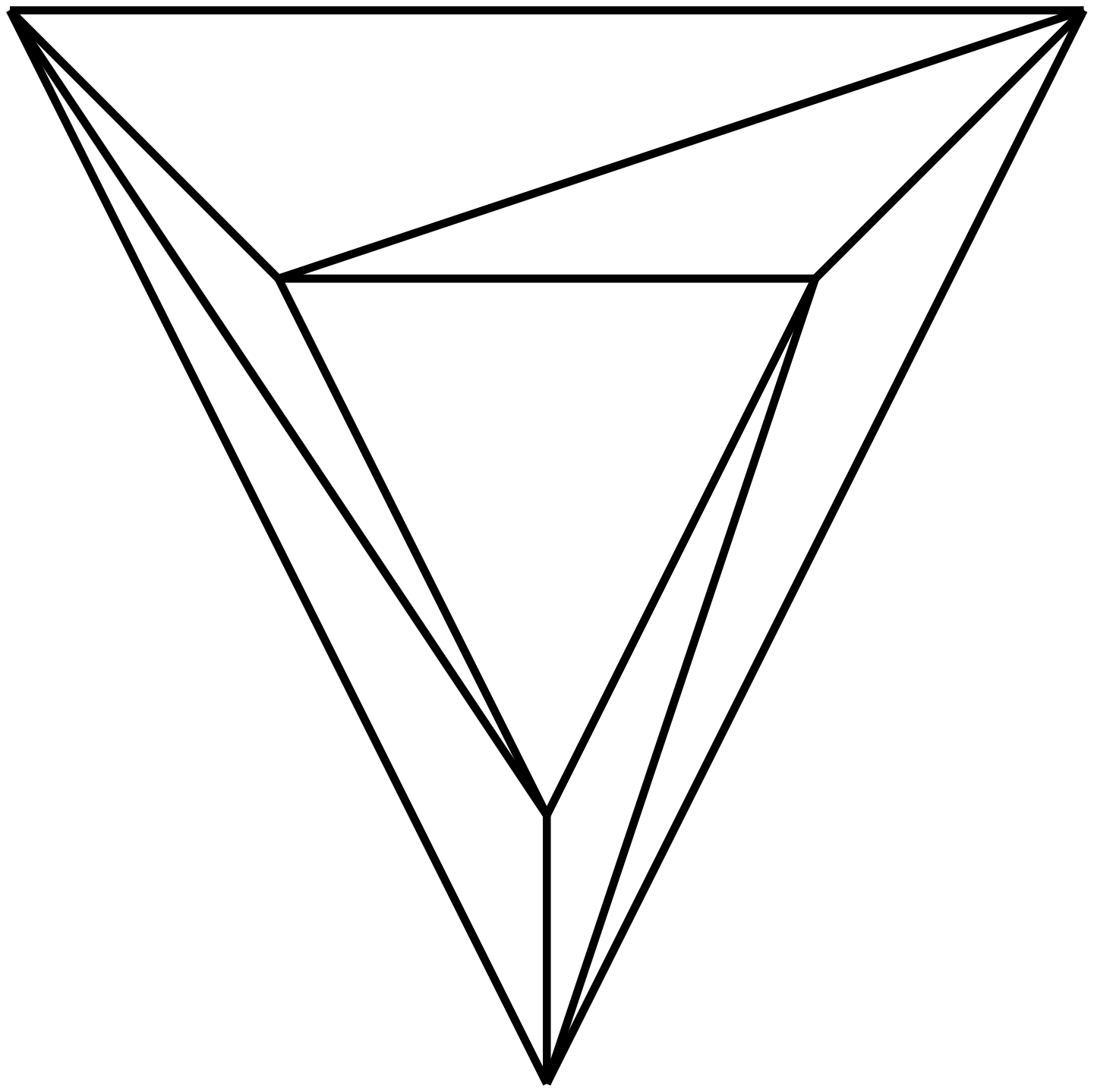}
 \caption{ $\Delta$}
\label{delta}
\end{minipage}
\hspace{0.1cm} 
\begin{minipage}{0.45\linewidth}
\centering
\includegraphics[keepaspectratio=true,  width=38mm, height=47mm]{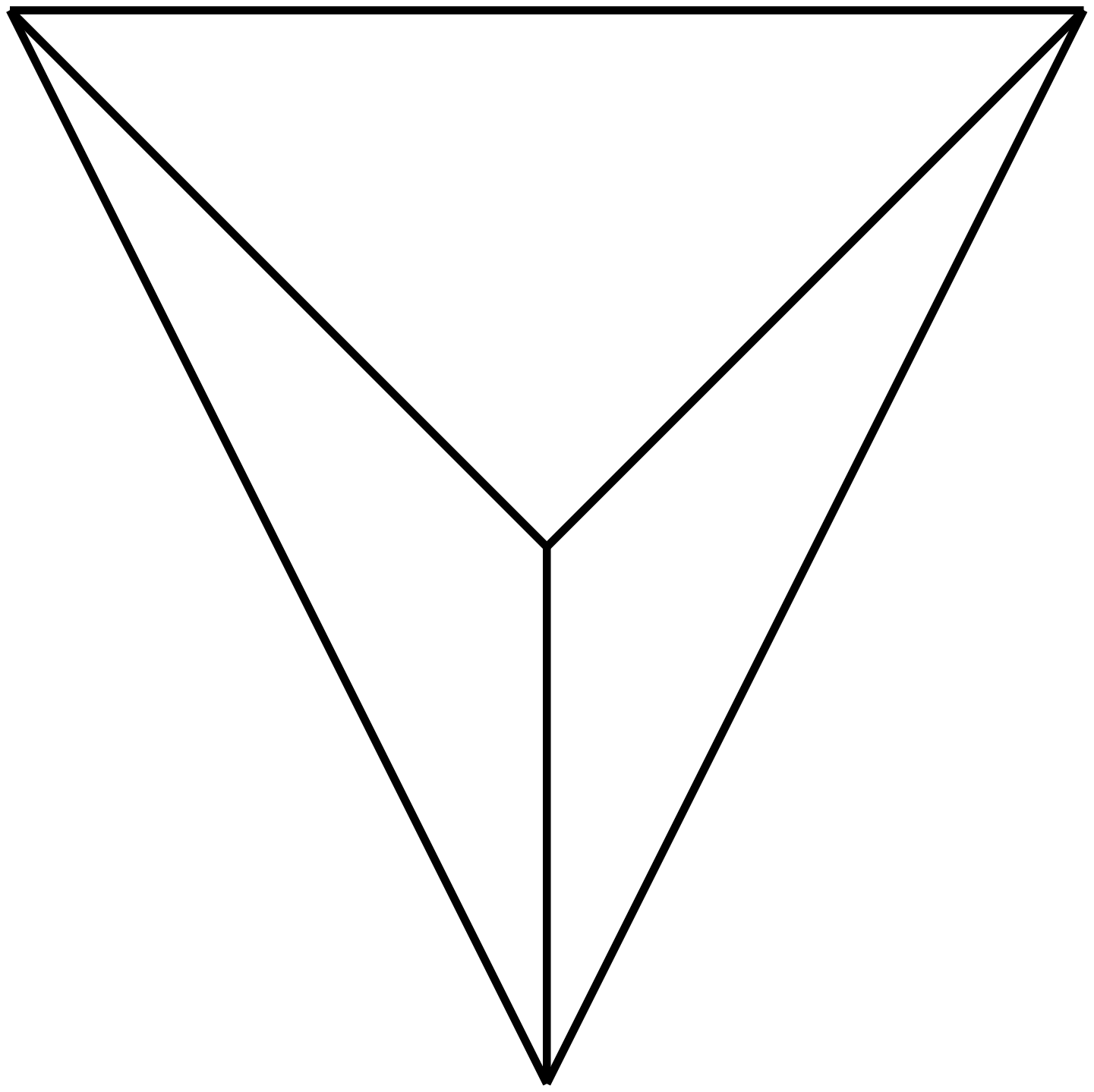}
 \caption{$\Delta''$}
 \label{sigma}
\end{minipage}
\end{figure}
\begin{figure} 
\begin{minipage}{0.45\linewidth}
\centering
\includegraphics[keepaspectratio=true,  width=38mm, height=47mm]{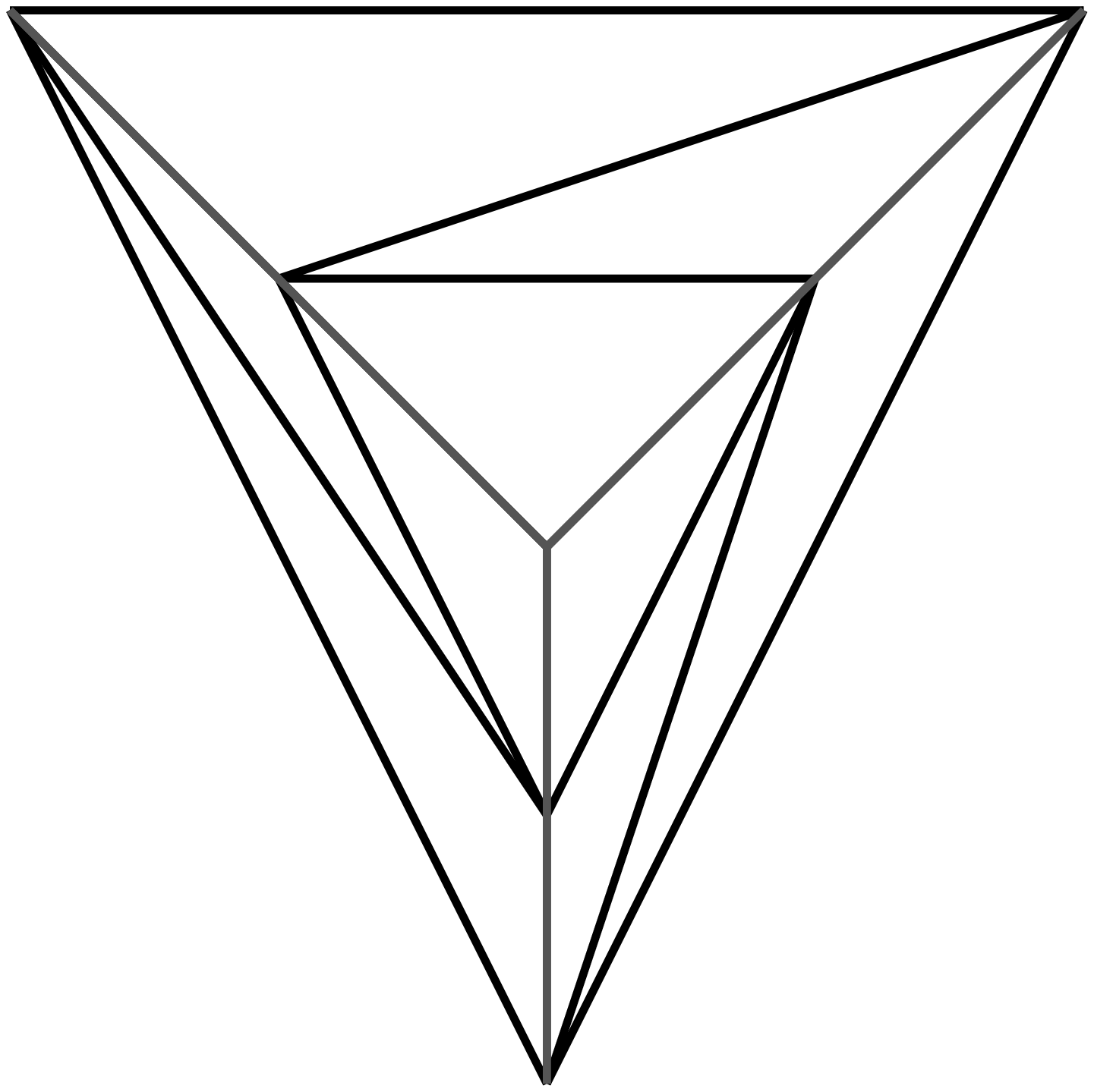}
 \caption{$\Delta'$}
 \label{delta1}
\end{minipage}
\begin{minipage}{0.45\linewidth}
\centering
\includegraphics[keepaspectratio=true,  width=38mm, height=47mm]{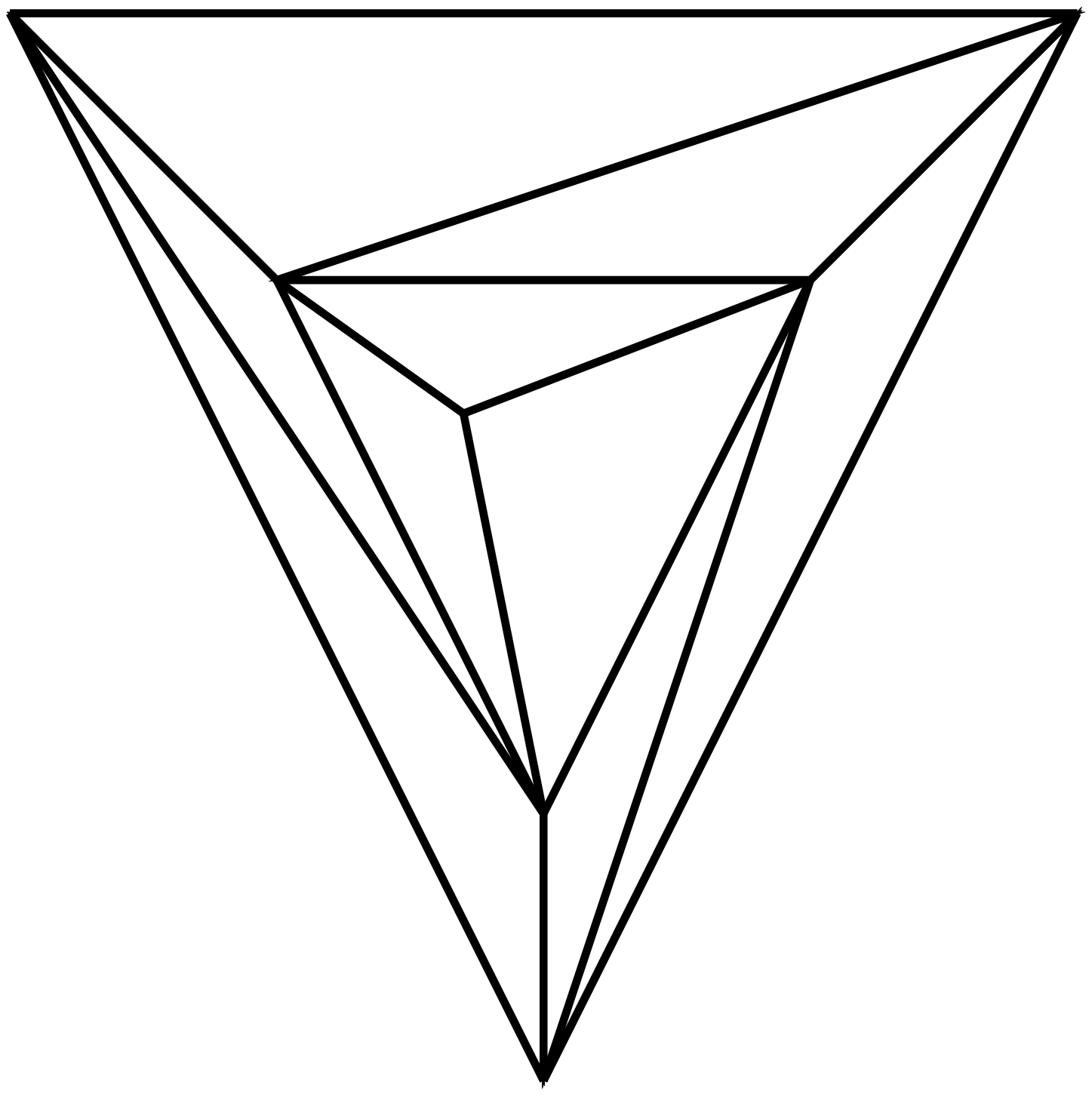}
 \caption{$\widetilde\Delta'$}
 \label{delta2}
\end{minipage}
\end{figure}
\end{exm}

For any subdivision $\Delta'$ of $\Delta$, there is a tautological map of 
chain complexes
\[
\RJ(\Delta) \stackrel{v}{\longrightarrow}\RJ(\Delta'), 
\]
where if $\gamma \in \Delta_k$ is subdivided into $\gamma'_1,\ldots, \gamma'_m \in \Delta'_k$, $v$ is induced by the map
\[
\bigoplus\limits_{\gamma \in \Delta_k}R \stackrel{v}{\longrightarrow} \bigoplus\limits_{\gamma' \in \Delta'_k}R, \mbox{ which sends }1_\gamma \mbox{ to }\sum_{i=1}^m1_{\gamma'_i}.
\] 
In general $v$ will have both a kernel and a cokernel. 
\begin{lem}\label{4terms}
For a simple subdivision $\Delta'$ of $\Delta$, $\ker(v)$ is
supported on $\partial(\Delta'')$, and $\coker(v)$ is supported on $(\Delta'')^0$. If $\Delta'$ is split then $\ker(v) = 0$. 
\end{lem}
\begin{proof}
The faces of $\Delta \setminus \sigma$ and $\Delta' \setminus \Delta''$ are identical, so for $\gamma \in \Delta \setminus \sigma = \Delta' \setminus \Delta''$ the ideal $J(\gamma)$ is also the same. In particular, both $\ker(v)$ and $\coker(v)$ are nonzero only on
$\Delta''$. No face of $\Delta$ meets $(\Delta'')^0$ save $\sigma$
itself, which $v$ maps to the sum of $k$-faces of $\Delta''$. Thus 
\[\coker (v)=
\begin{cases}
&R^{|\Delta''_k|}/R, \hskip 50pt \hbox{in degree $k$},\\
&\bigoplus\limits_{\gamma\in \Delta''_i} R/J(\Delta')_\gamma, \hskip 20pt\hbox{for all $i<k$,}
\end{cases}
\]
and $\coker(v)$ is nonzero only on the interior of $\Delta''$. For the kernel, if 
$\gamma\in \Delta\setminus \sigma$, then 
\[
J(\Delta)_\gamma=J(\Delta')_\gamma,
\]
so the kernel
can only be nonzero on $\partial(\sigma) = \partial(\Delta'').$ If $\Delta'$ is split, 
then again 
\[
J(\Delta)_\gamma=J(\Delta')_\gamma
\]
for $\gamma \in \partial(\sigma)$, hence $\ker(v) = 0$. 
\end{proof}

\begin{prop}\label{diagram}
A split subdivision $\Delta'$ of $\Delta$ gives rise to 
a  short exact sequence of complexes: a commuting diagram where the columns are exact, and the rows form complexes:
\[
\xymatrix{ 
&~&  0 \ar[d] & 0\ar[d]   & 0 \ar[d] \\
&\RJ(\Delta): 0 \ar[r] &\bigoplus\limits_{\sigma \in \Delta_k}\hskip -5ptR \ar[r]^{\hskip -7pt\partial_k}  \ar[d]^{v_k} & \bigoplus\limits_{\tau \in \Delta^0_{k-1}}\hskip -7pt R/J_\tau \ar[r]^{\partial_{k-1}}  \ar[d]^{v_{k-1}}  & \bigoplus\limits_{\psi \in \Delta^0_{k-2}}\hskip -7pt R/J_\psi \ar[r]^{~~~~~~\partial_{k-2}} \ar[d]^{v_{k-2}} & \cdots \\
&\RJ(\Delta'): 0 \ar[r] &\bigoplus\limits_{\sigma \in \Delta'_k}\hskip -5ptR \ar[r]^{\hskip -7pt\partial_k}  \ar[d] & \bigoplus\limits_{\tau \in \Delta'^0_{k-1}}\hskip -7pt R/J_\tau\ar[r]^{\partial_{k-1}}  
\ar[d]  & \bigoplus\limits_{\psi \in \Delta'^0_{k-2}}\hskip -7pt R/J_\psi\ar[r]^{~~~~~~\partial_{k-2}} \ar[d] & \cdots \\
&Q: 0 \ar[r]& \coker(v_k) \ar[r]^{\hskip -7pt\partial_k}  \ar[d] & \coker(v_{k-1})  \ar[r]^{\partial_{k-1}}  
\ar[d]  & \coker(v_{k-2}) \ar[r]^{~~~~~~\partial_{k-2}} \ar[d] & \cdots \\
 &~&  0 & 0  & 0 
}
\]
\end{prop}
\begin{proof}
Follows from Lemma~\ref{4terms} and Definition~\ref{splitSD}.
\end{proof}
\vskip .05in
\begin{cor}~\label{sesS}
If $\Delta'$ is a split subdivision of $\Delta$ and 
$H_{k-1}(\RJ(\Delta))=0$, then 
\[ 
0 \longrightarrow S^r(\widehat \Delta) \longrightarrow S^r(\widehat \Delta') \longrightarrow H_k(Q) \longrightarrow 0 
\]
is an exact sequence, so 
\[ 
\dim S^r(\widehat \Delta')_d= \dim S^r(\widehat \Delta)_d+ \dim
H_k(Q)_d. 
\]
\end{cor}
\begin{proof} Recall that a short exact sequence of complexes yields a long exact sequence in 
homology
\begin{equation*}
 \rightarrow H_{i\!+\!1}(Q) \rightarrow H_{i}(\RJ(\Delta)) 
\rightarrow H_{i}(\RJ(\Delta')) \rightarrow H_{i}(Q) 
\rightarrow H_{i\!-\!1}(\RJ(\Delta)) \rightarrow 
\end{equation*}
and the result follows.
\end{proof}
\subsection{Main result} 
We are now ready to explore why split subdivisions are special.
\begin{thm}\label{splitS}
If $\Delta'$ is a split subdivision of $\Delta$, then for $i < k$, 
\[
\coker(v_{i})\simeq \RJ(\Delta'')_i,
\]
and in degree $k$
\[
\coker(v_k) \simeq (\bigoplus\limits_{\sigma \in \Delta''_k}R)/R,
\]
with the quotient map is via the diagonal.
\end{thm}
\begin{proof}
Since the subdivision is split, the complexes 
$\RJ(\Delta)$ and $\RJ(\Delta')$ agree on the common faces, which are
all faces save those in the interior of $\Delta''$. This means
that the vertical maps in the double complex in Proposition~\ref{diagram} are either
the identity or zero on each individual term of the direct sums. Hence the 
cokernel of $v_i$ is simply $\RJ(\Delta'')_i$. The exception
to this is on $\sigma$, which maps via the identity to each $k$-face of
$\Delta''$.
\end{proof}
\begin{thm}\label{allFree}
If $\Delta'$ is a split subdivision of $\Delta$ and 
both $S^r(\widehat \Delta)$ and $S^r(\widehat \Delta'')$ are free, 
then $S^r(\widehat \Delta')$ is free.
\end{thm}
\begin{proof}
By Theorem~\ref{splitS}, $H_k(Q) \simeq S^r(\widehat \Delta'')$, modulo
constant splines. The result follows from the long exact sequence in homology associated with the short exact sequence of complexes in 
Proposition~\ref{diagram} coupled with 
Theorem 4.10 of~\cite{s1}.
\end{proof}

\section{Applications and computations}\label{computations}
In this section, we apply the results of \S~\ref{homology}  to various subdivisions 
$\Delta''$ of $\sigma$ such that $S^r(\widehat \Delta'')$ is free. Our starting point is
recent progress in obtaining the dimension of multivariate splines of arbitrary degree and
smoothness on the so-called Alfeld split $A(T_k)$ of an $k$-dimensional simplex $T_k$ in $\R^k$, 
which is a higher dimensional analog of the  Clough-Tocher split of a triangle, see~\cite{a1} or Sections~18.3, 18.7 of~\cite{ls}. The split $A(T_k)$ is 
obtained from a single simplex $T_k$ by adding a single  interior vertex $u$, 
and then coning over the boundary of $T_k$.  A formula for the dimension of the space $S_d^r( A(T_k))$ 
of splines of smoothness $r$ and polynomial degree at most $d$ on $A(T_k)$ was conjectured in~\cite{fs}, 
and proved in \cite{s3}, where it was also shown that the module is free. 
\begin{thm}\label{alfeld} \cite{s3} Let $A(T_k)$ be the Alfeld split of an $k$-simplex $T_k$ in $\R^k$. Then
\[ 
\dim S^r_d(A(T_k)) = {d+k \choose k} +  A(k,d,r),\]
where
\[A(k,d,r):=\begin{cases} 
k \binom{d+k-\frac{(r+1)(k+1)}{2}}{k},\hskip 72pt \mbox{ if r is odd},\\
\binom {d+k-1-\frac{r(k+1)}{2}}{k}+\cdots + \binom{d-\frac{r(k+1)}{2}}{k}, \mbox{ if r is even}. 
\end{cases}
\] 
Moreover, the associated module of splines $S^r( \widehat A(T_k))$ is free for any $r$.
\end{thm}
The Alfeld split can be further refined to obtain other splits useful in applications, in particular
for constructing macro-element spaces. Such constructions are not possible unless the exact dimension 
of the spline space of interest is known.  In this section we concentrate on $k$-dimensional analogs
of two known refinements of Alfeld splits in $\R^2$ and $\R^3$. The first one is known in $\R^2$ as the 
Powell-Sabin split of a triangle, see Figure~\ref{2dfacet}, and Sections~6.3, 7.3, 8.4 of~\cite{ls}, and the references therein. 
Its  analog  in $\R^3$ has been called both  
Worsey-Farin  and Clough-Tocher, see e.g. Section~18.4 and~18.8 of~\cite{ls} with the references therein.
In order to eliminate any ambiguity, we introduce 
the following definition in $\R^k$.

\begin{figure}
\vskip 15pt
\begin{minipage}{0.45\linewidth}
\centering
 \includegraphics[keepaspectratio=true,  width=50mm, height=57mm]{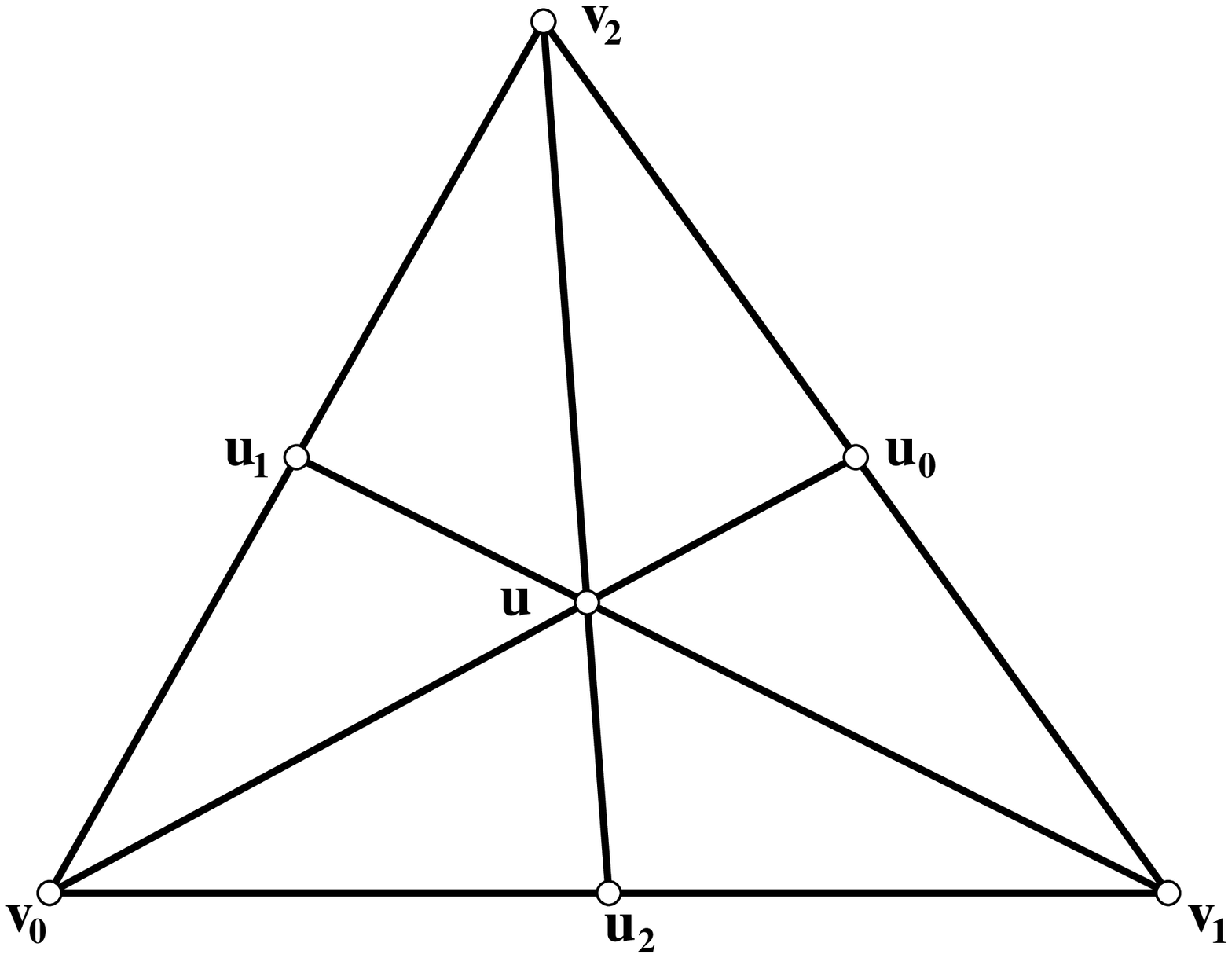}
 \caption{$F(T_2)$}
\label{2dfacet}
\end{minipage}
\hspace{0.5cm} 
\begin{minipage}{0.45\linewidth}
\centering
\includegraphics[keepaspectratio=true,  width=50mm, height=57mm]{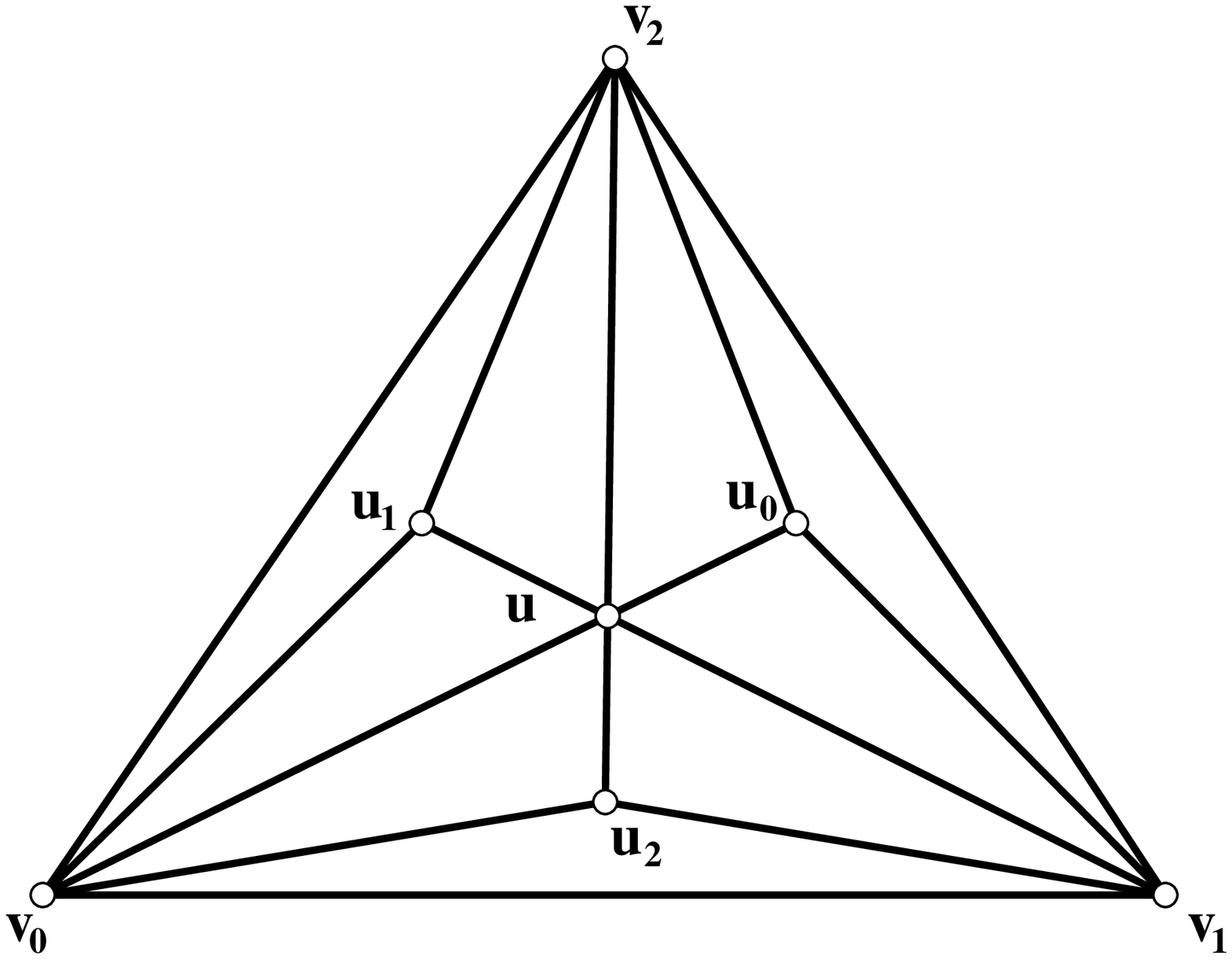}
\vskip 12pt
 \caption{$AA(T_2)$ }
 \label{2daa}
\end{minipage}
\end{figure}

\begin{defn}\label{WF} For a full-dimensional $k$-simplex $T_k:=[v_0,v_1,\dots,v_k] \subseteq \R^k$, let $A(T_k)$ be the Alfeld split
with the interior vertex $u$. The facet split $F(T_k)$  is obtained by further subdividing~$A(T_k)$ as follows. For 
each $i=0,\dots,k$, let~$F_i$ be the facet of~$T_k$ opposite vertex~$v_i$. Let  $u_i$ be the point strictly interior to $F_i$ 
and collinear with $v_i$ and $u$. Each $u_i$ induces a $(k-1)$-dimensional  Alfeld
split $A(F_i)$ of $F_i$. Finally, cone $u$ over  $A(F_i)$ forming a pyramid $P_i$ in $\R^k$. The collection
of $k+1$ pyramids $P_i$  is the facet split $F(T_k)$.
\end{defn}
Note that if $u$ is the barycenter of $T_k$, and each $u_i$ is the barycenter of $F_i$, then the collinearity condition is satisfied.
$F(T_k)$ consists of $k^2+k$ simplices, has one interior vertex
$u$, and $2k+2$ boundary vertices.
Figure~\ref{fem} shows an unfinished (for clarity) $F(T_3)$, where Definition~\ref{WF} was carried out for $i=0$ only, thus splitting 
the facet $F_0$ only and forming a pyramid $P_0$ while keeping the remaining three subtetrahedra of $A(T_3)$ intact. 
\begin{prop}\label{pyramid} Let $T_{k-1}$ be  a $(k-1)$-simplex in $\R^k$, and let $P_k$ be a pyramid in $\R^k$ obtained by forming a cone
with the base  $A(T_{k-1})$ and an arbitrary vertex not coplanar with $\aff(T_{k-1})$. Then 
\[ 
\dim S^r_d(P_k) = {d+k \choose k} +  P(k,d,r),\]
where
\[P(k,d,r):=\begin{cases} 
(k-1) \binom{d+k-\frac{(r+1)k}{2}}{k},\hskip 32pt \mbox{ if r is odd}\\
\binom {d+k-1-\frac{rk}{2}}{k}+\cdots + \binom{d+1-\frac{rk}{2}}{k}, \mbox{ if r is even}. 
\end{cases}
\] 
Moreover, the associated module of splines $S^r( \widehat P_k)$ is free for any $r$.
\end{prop}
\begin{proof}
Since  a pyramid $P_k\subseteq {\mathbb R}^{k}$ is a cone over the Alfeld split 
$A_{k-1} \subseteq \mathbb{R}^{k-1}$ of a tetrahedron $T_{k-1}$, 
Theorem~\ref{alfeld} yields
\[
\dim S^r_d(P_k)=\sum_{i=0}^d\Bigg[{i+k-1 \choose k-1} +  A(k-1,i,r)\Bigg]={d+k \choose k} +  P(k,d,r),
\]
and the proof is complete.
\end{proof}
The second refinement of  interest is the $k$-dimensional analog of the 
so-called double Clough-Tocher split in $\R^2$, see Figure~\ref{2daa}, and Section 7.5 of~\cite{ls} along with the references therein. 
We shall call the new refinement
the \textit{double Alfeld} split to emphasize the multivariate nature. 

\begin{defn}\label{double} For an $k$-simplex $T_k:=[v_0,v_1,\dots,v_k] \subseteq \R^k$, let $A(T_k)$ be the Alfeld split
with the interior vertex $u$. The double Alfeld split $AA(T_k)$  is obtained by further subdividing~$A(T_k)$ as follows for 
each $i=0,\dots,k$.  Let~$F_i$ be the facet of~$T_k$ opposite vertex~$v_i$. Let  $u_i$ be a point strictly interior to 
the simplex $T_k^i:=[u, F_i]$ and collinear with $v_i$ and $u$. Each $u_i$ induces an  Alfeld
split $A(T_k^i)$ of $T^i_k$. The collection
of $k+1$ Alfeld splits $A(T_k^i)$  is the double Alfeld split $AA(T_k)$.
\end{defn}

\begin{figure}
\begin{minipage}{0.45\linewidth}
\centering
 \includegraphics[keepaspectratio=true,  width=50mm, height=57mm]{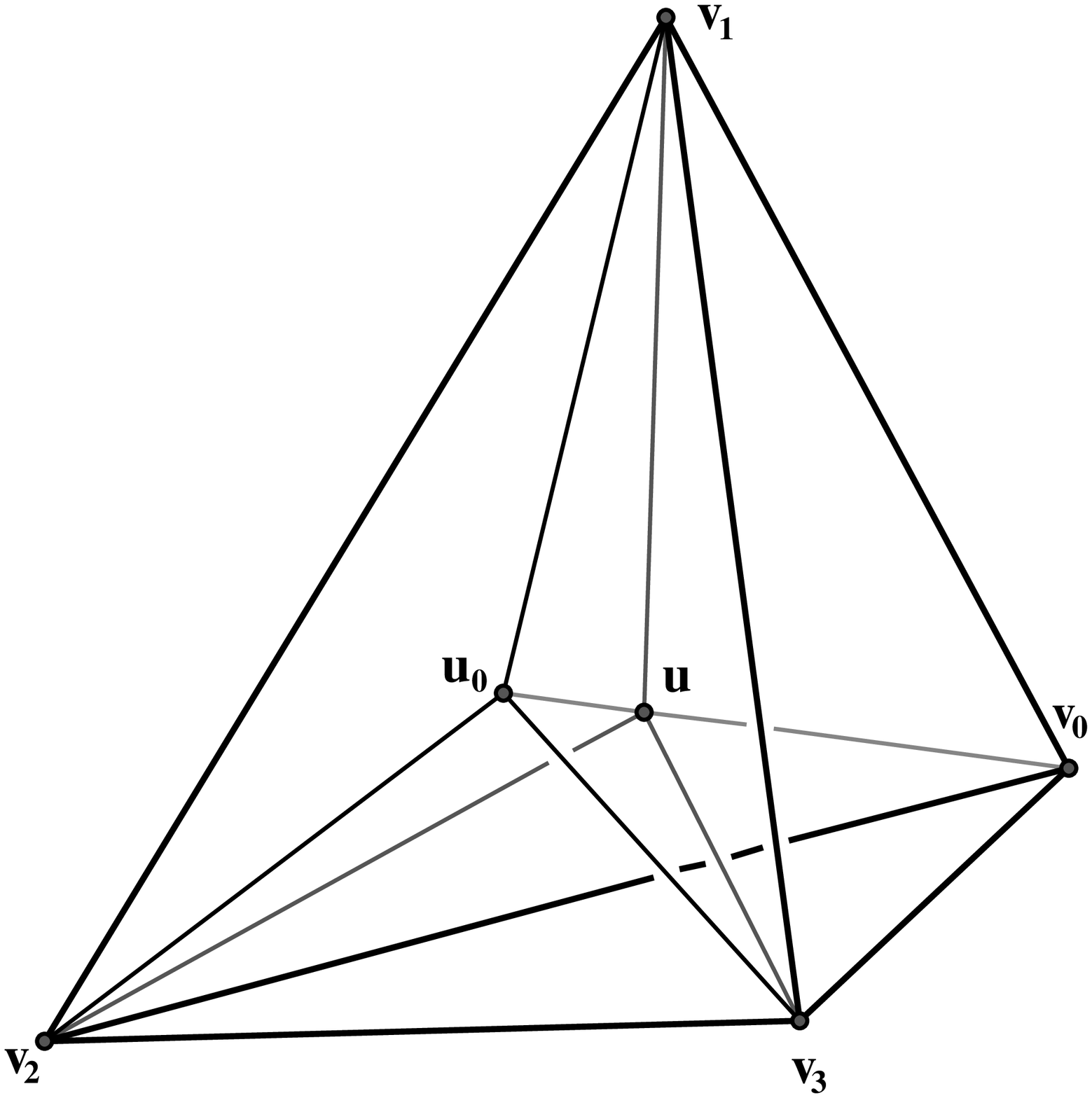}
 \caption{A part of $F(T_3)$}
\label{fem}
\end{minipage}
\hspace{0.5cm} 
\begin{minipage}{0.45\linewidth}
\centering
\includegraphics[keepaspectratio=true,  width=50mm, height=67mm]{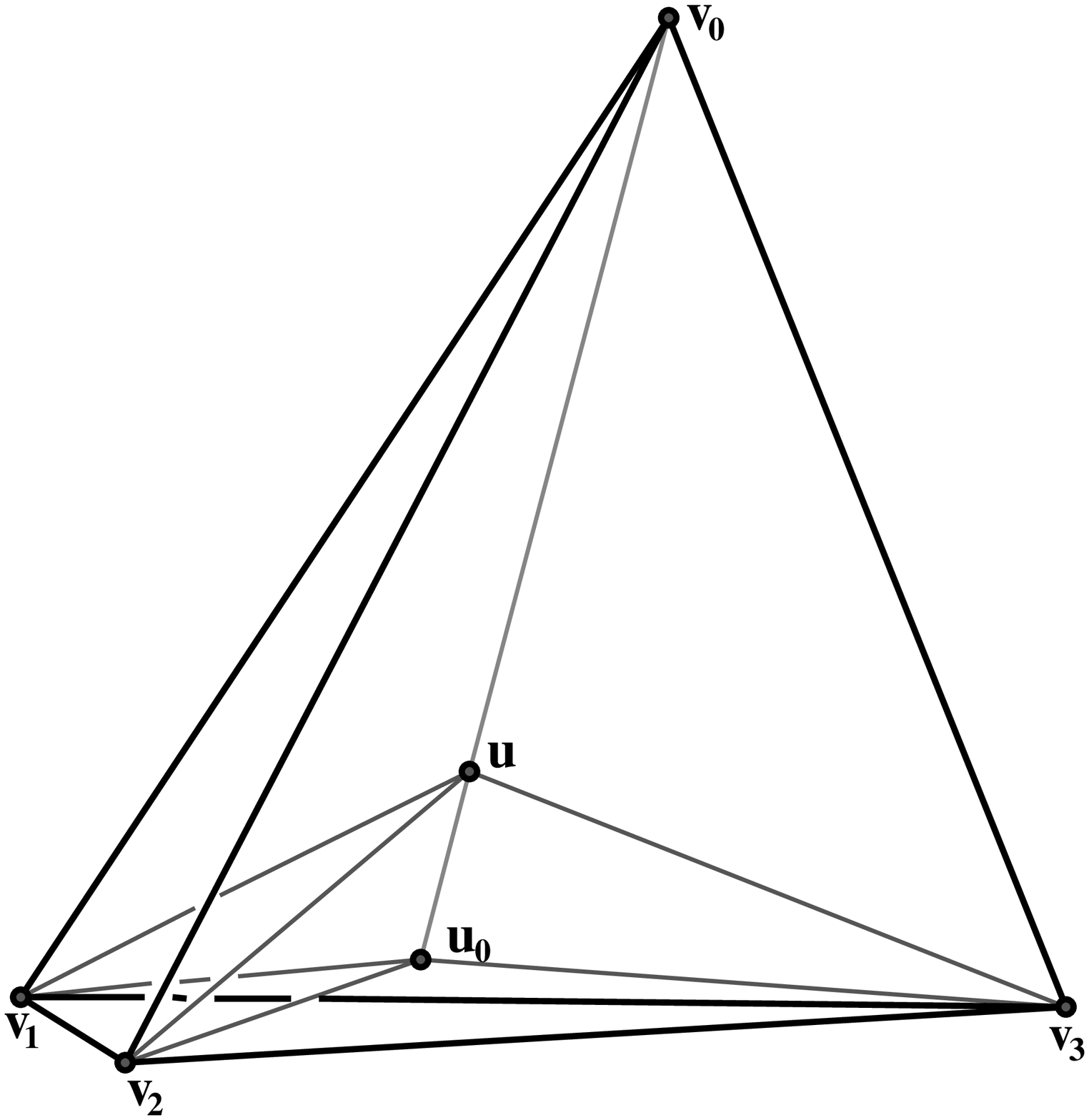}
 \caption{A part of  $AA(T_3)$}
 \label{fem1}
\end{minipage}
\end{figure}

Note that if $u$ is the barycenter of $T_k$, and each $u_i$ is the barycenter of $T_k^i$, then the collinearity 
condition is satisfied. $AA(T_k)$ consists of $(k+1)^2$ simplices, has $k+2$ interior vertices,
 and $k+1$ boundary vertices.  Figure~\ref{fem1}
shows an unfinished (for clarity) $AA(T_3)$,  where Definition~\ref{double} was carried out for $i=0$ only, thus splitting 
the subterahedron $T_3^0$ only while keeping the remaining three subtetrahedra of $A(T_3)$ intact.
We are now ready to prove the main result of this section.
\begin{thm}\label{main} For an $k$-simplex $T_k$ in $\R^k$ let $F(T_k)$ and $AA(T_k)$ be
the associated facet and double Alfeld splits as in Definition~\ref{WF} and~\ref{double}. Then
\begin{align*}
\dim S^r_d(F(T_k)) = &{d+k \choose k} + A(k,d,r)+(k+1) P(k,d,r),\\
\dim S^r_d(AA(T_k)) = &{d+k \choose k} + (k+2) A(k,d,r),
\end{align*}
where $A(k,d,r)$ and $P(k,d,r)$ are as in Theorem~\ref{alfeld} and Proposition~\ref{pyramid}, respectively.
Moreover, the associated modules of splines $S^r( \widehat F(T_k))$ and  $S^r( \widehat {AA}(T_k))$ are free for any $r$.
\end{thm}
\begin{proof} We start by subdividing the single simplex $\sigma:=[u,v_1,v_2\dots,v_k]$ in $A(T_k)$, as in Definition~\ref{splitSD}. 
In the case of the facet split, $\Delta_F''$ is the pyramid $P_k$  described in Theorem~\ref{pyramid}. Figure~\ref{fem} demonstrates  
the 3D case, where the point $u_0$ is placed in the face $[v_1,v_2,v_3]$.  For the 
double Alfeld split, $\Delta_{AA}''$ is the Alfeld split of $\sigma$. Figure~\ref{fem1} depicts the 3D case, where the point $u_0$ is placed in the interior of the tetrahedron
$[u, v_1, v_2, v_3]$.
In either case, due to the collinearity conditions on $u$, $v_0$ and $u_0$,
the resulting subdivisions $\Delta_F'$  and $\Delta'_{AA}$ are simple and split as in Definition~\ref{splitSD}. 
Then by Corollary~\ref{sesS} and Theorem~\ref{splitS},
we obtain 
\begin{align*}
\dim S^r_d(\Delta'_F) = &\dim S^r_d(A(T_k))+ \dim S^r_d(P_k)-\binom{d+k}{k},\\
\dim S^r_d(\Delta'_{AA}) = &\dim S^r_d(A(T_k))+ \dim S^r_d(A(T_k))-\binom{d+k}{k}.
\end{align*}
Moreover, since by Theorem~\ref{allFree} the associated modules of splines $S^r( \widehat \Delta'_F)$ and  $S^r( \widehat\Delta'_{AA} )$ are free, we 
can apply the same technique to the next simplex $[u,v_0,v_2,v_3,\dots,v_k]$ in the intermediate subdivision $\Delta'_F$ or $\Delta'_{AA}$, and so on. Using the dimension formulae in Theorem~\ref{alfeld} and 
Proposition~\ref{pyramid} completes the proof.
\end{proof}
 \section{Remarks}
 \begin{rmk}\label{bases} As a consequence of the split of the module
of splines on $\Delta'$, we also obtain explicit bases for $S^r(\Delta')$ essentially as a union of generators for $S^r(\Delta)$ and $S^r(\Delta'')$. The generators are not   as useful in applications as more traditional B-spline or \B~bases, and an efficient conversion algorithm  is an open computational problem. 
 \end{rmk}
  \begin{rmk}\label{partial} The proof of Theorem~\ref{main} holds for partial facet and double Alfeld splits, i.e. for the case where not every
  tetrahedron in $A(T_k)$ is subdivided. Such partial subdivisions are useful  in the context of  boundary finite elements. 
 \end{rmk}
 \begin{rmk}\label{collinearity} As  Example~\ref{r1} demonstrates, 
the requirement of the collinearity in both Definition~\ref{WF} and~\ref{double} can be omitted for $r=1$.
 \end{rmk}

 \begin{rmk}\label{other} The splitting method of \S~\ref{homology} can be applied to more subdivisions, including those of a simplex in $\R^k$.
 We focused on two well-known splits that do not require consideration of multiple cases stemming from exact geometry. We also note that the freeness of the modules of splines involved in the splitting method is a sufficient but not a necessary condition. We are investigating extensions of the  results here to other situations.
 \end{rmk}
\noindent {\bf Acknowledgements} 
Computations in the Macaulay2 package of Grayson and Stillman (available at {\tt http://www.math.uiuc.edu/Macaulay2}) and in Alfeld's spline software (available at {\tt http://www.math.utah.edu/$\sim$pa}) were essential to this work. We also thank the Mathematische Forschungsinstitut Oberwolfach, where our collaboration began.

\renewcommand{\baselinestretch}{1.0}
\small\normalsize 

\bibliographystyle{amsalpha}

\end{document}